\documentclass[12pt]{amsart}
\usepackage{amsmath,amsthm,amscd,amsfonts,amssymb}
\usepackage{latexsym}
\usepackage{comment}

\begin{document}

\newtheorem{theorem}{Theorem}[section]
\newtheorem{prop}[theorem]{Proposition}
\newtheorem{lemma}[theorem]{Lemma}
\newtheorem{cor}[theorem]{Corollary}

\title{On a Theorem of Scott and Swarup}
\thanks{Research partly supported by Alfred P. Sloan
Doctoral Dissertation Fellowship, Grant No. DD 595 \\ 
AMS subject classification =   20F32, 57M50 \\ This is a corrected version of the paper \cite{mahan-scottswar}.   The original references
to preprints,
many of which are published now, have been retained. Extra references are added.}

\author{Mahan Mitra}

\begin{abstract}
 Let $1 \rightarrow H \rightarrow G \rightarrow \mathbb{Z} 
\rightarrow 1$ be an exact sequence of hyperbolic groups
induced by a fully irreducible automorphism  $\phi$ of the free group
$H$. 
Let  $H_1 ( \subset H)$ be  a  finitely
generated distorted 
subgroup of $G$.
Then  $H_1$ is
of finite index in $H$. This is
an analog of a Theorem of Scott and Swarup for surfaces
in hyperbolic 3-manifolds.
\end{abstract}

\maketitle

\section{Introduction}

In \cite{scottswar} , Scott and Swarup prove the following theorem:

\medskip

{\bf Theorem} \cite{scottswar} {\it
Let $1 \rightarrow H \rightarrow G \rightarrow  {\mathbb{Z}} \rightarrow 1$
be an exact sequence of hyperbolic groups
induced by a pseudo Anosov  diffeomorphism of a
closed surface with
fundamental group $H$.
  Let $H_1$ be a finitely generated subgroup of infinite index
 in 
$H$. Then $H_1$ is quasiconvex in $G$.}

\medskip

In this paper we derive an analogous result for free groups (see Section 2
below or \cite{BFH} \cite{BFHa} \cite{Gromov2} for definitions).

We note at the outset that {\it hyperbolic} stands for two notions.
When qualifying manifolds, they indicate spaces of constant
 curvature equal to -1. When qualifying groups or metric spaces,
we use {\it hyperbolic} in the sense of Gromov \cite{Gromov}.
It will be clear from the context which of these meanings is 
relevant.

\medskip

{\bf Theorem \ref{main} } {\it
Let $1 \rightarrow H \rightarrow G \rightarrow  {\mathbb{Z}} \rightarrow 1$
be an exact sequence of
hyperbolic groups
induced by an aperiodic fully irreducible automorphism of the
free group $H$. Let $H_1$ be a finitely generated subgroup of infinite index
 in 
$H$. Then $H_1$ is quasiconvex in $G$.}

\medskip

In fact the methods of this paper can be used to give a new proof
of the Theorem of Scott and Swarup mentioned above. We sketch this proof 
for closed surfaces first. 
Let $M$ be a closed hyperbolic 3-manifold fibering over the circle with 
fiber $F$. Let $\widetilde F$ and $\widetilde M$ denote the universal
covers of $F$ and $M$ respectively. Then $\widetilde F$ and $\widetilde M$
are quasi-isometric to ${\mathbb{H}}^2$ and ${\mathbb{H}}^3$ respectively. Now let
${{\mathbb{D}}^2}={\mathbb{H}}^2\cup{\mathbb{S}}^1_\infty$ and 
${{\mathbb{D}}^3}={\mathbb{H}}^3\cup{\mathbb{S}}^2_\infty$
denote the standard compactifications. In \cite{CT} Cannon and Thurston
show that the usual inclusion $i$
of $\widetilde F$ into $\widetilde M$
extends to a continuous map $\hat{i}$
from ${\mathbb{D}}^2$ to ${\mathbb{D}}^3$. Cannon and Thurston further 
show that $\hat{i}$ identifies precisely those pairs of points which
are boundary points of an ending lamination. Since a leaf of the
stable (or unstable) lamination is dense in in the whole lamination,
it cannot be carried by a (perhaps immersed) proper sub-surface (one can
see this, for instance, by using the fact that surface groups are LERF
\cite{scott} ). The subgroup corresponding to the fundamental group
of such a subsurface must therefore be quasiconvex in $G$.

\medskip

This idea goes through for free groups. We give a brief sketch for
aperiodic automorphisms. In this case, Bestvina, Feighn and Handel
\cite{BFH} have shown that any leaf of the stable (or unstable)
lamination `fills' $H$, i.e. it cannot be carried by a finitely
generated subgroup $H_1$ of infinite index in $H$. We combine this with the
description of boundary identifications given in \cite{mitra2}
to show that no pair of points on the boundary of $H_1$ are identified.
Thus $H_1$ must be quasiconvex in $G$. 

\medskip

\section{Ending Laminations}

Let $G$ be a hyperbolic group in the sense of Gromov \cite{Gromov} . Let $H$
be a hyperbolic subgroup of $G$. Choose a finite
generating set of $G$ containing a finite generating set of $H$. Let
$\Gamma_G$ and $\Gamma_H$ be the Cayley graphs of $G$, $H$ with respect to 
these generating sets. Let $i : \Gamma_H \rightarrow \Gamma_G$ denote the
inclusion map.

\medskip

{\bf Definition :} \cite{Gromov2} \cite{farb}
{\it If $i: \Gamma_H \rightarrow \Gamma_G$ 
be an embedding of the Cayley
graph of $H$ into that of $G$, then the {\it distortion} function is
given by
\begin{center}
$disto(R) = Diam_{\Gamma_H}({\Gamma_H}{\cap}B(R))$,
\end{center}
where $B(R)$ is the ball of radius $R$ around $1\in{\Gamma_G}$.}

\medskip

If $H$ is quasiconvex in $G$ the distortion function is linear and we shall
refer to $H$ as an undistorted subgroup. Else, $H$ will be termed
distorted.

For distorted subgroups, the distortion information is encoded in a
certain set of ending laminations defined below.

\medskip

{\bf Definition :} {\it If $\lambda$ is a geodesic segment in $\Gamma_H$
then $\lambda^r$, a {\bf geodesic realization} of $\lambda$, is a geodesic
in $\Gamma_G$ joining the end-points of $i({\lambda})$. }

Now consider sequences of geodesic segments $\lambda_i \subset \Gamma_H$
such that $1\in{\lambda_i}$ and
${\lambda_i^r}\cap{B(i)} = \emptyset$, where $B(i)$ is the
ball of radius $i$ around $1 \in \Gamma_G$. Take all bi-infinite
subsequential limits (in the Hausdorff topology) of all such sequences
$\{ \lambda_i \}$ and denote this set by $\Sigma$.

Let $t_h$ denote left translation by $h \in H$. Let $\widehat{\Gamma_H}$
and $\widehat{\Gamma_G}$ denote the Gromov compactifications of $\Gamma_H$
and $\Gamma_G$ respectively. Further let $\partial{\Gamma_H}$ and
$\partial{\Gamma_G}$ denote the boundaries of $\Gamma_H$
and $\Gamma_G$ respectively  \cite{Gromov} .

\medskip

{\bf Definition :} {\it The set of ending laminations 
$\Lambda = \Lambda (H,G)$ is given by 
\begin{center}
$\Lambda = \{ (p,q) \in \partial{\Gamma_H} \times \partial{\Gamma_H} | 
p \neq q$  and $p,q$ are the end-points of ${t_h}({\lambda})$ for some 
$\lambda \in \Sigma \}$
\end{center} }

\medskip

\begin{lemma}
$H$ is quasiconvex in $G$ if and only if $\Lambda = \emptyset$
\label{qccrit}
\end{lemma}

{\bf Proof :} Suppose $H$ is quasiconvex in $G$. Then any geodesic realization
$\lambda^r$ of a geodesic segment $\lambda \subset \Gamma_H$ lies in a 
bounded neighborhood of $\Gamma_H$ and hence of $\lambda$ as $H$ is
hyperbolic. Hence $\Lambda = \emptyset$ .

Conversely, if $H$ is not quasiconvex in $G$, there exist 
$\lambda_i \subset \Gamma_H$ and $p_i \in \lambda_i$ such that
$\lambda_i^r \cap B_{p_i} (i) = \emptyset$, where $B_{p_i} (i)$ denotes
the ball of radius $i$ around $p_i$ in $\Gamma_G$. Translating by
$p_i^{-1}$ and taking subsequential limits, we get $\Sigma \neq \emptyset$
and hence $\Lambda \neq \emptyset$. $\Box$

\medskip

{\bf Definition :} {\it A {\bf Cannon-Thurston map} for the pair
$(H,G)$ is a map 
$\hat{i} : \widehat{\Gamma_H} \rightarrow \widehat{\Gamma_G}$ 
which is a continuous
extension of $i : \Gamma_H \rightarrow \Gamma_G$. }

\medskip

Note that if such a continuous extension exists, it is unique.
We get a simplified collection of ending laminations when a Cannon-Thurston
map exists.

\smallskip

{\bf Definition :} {\it $\Lambda_{CT} = 
 \{ (p,q) \in \partial{\Gamma_H} \times \partial{\Gamma_H} | 
p \neq q$  and $\hat{i} (p) = \hat{i} (q) \}$}.

\medskip

\begin{lemma}
If a Cannon-Thurston map exists, $\Lambda = \Lambda_{CT}$.
\label{equality}
\end{lemma}

{\bf Proof:}
Let $(p,q) \in \Lambda$. After translating by an element of $H$ if necessary
assume that a bi-infinite geodesic  $\lambda$ passing through $1$ has
$p, q$ as its end-points. By definition of $\Lambda$ there exist geodesic
segments $\lambda_i \subset \Gamma_H$ converging to $\lambda$ in the Hausdorff
topology such that $\lambda_i^r \cap B(i) = \emptyset$. Since a 
Cannon-Thurston map exists, there exists $z \in \partial{\Gamma_G}$
such that $\lambda_i^r \rightarrow z$ in the Hausdorff topology on
$\widehat{\Gamma_G}$ and $\hat{i} (p) = z = \hat{i} (q)$.
Hence $\Lambda \subset \Lambda_{CT}$.

Conversely, 
let $(p,q) \in \Lambda_{CT}$. 
After translating by an element of $H$ if necessary
assume that a bi-infinite geodesic  $\lambda$ passing through $1$ has
$p, q$ as its end-points. Choose ${p_i}, {q_i} \in \Gamma_H$ such that
$p_i \rightarrow p$ and $q_i \rightarrow q$. Let $\lambda_i$ denote
the subsegment of $\lambda$ joining ${p_i}, {q_i}$. Then $\lambda_i^r$
converges to $\hat{i} (p) = \hat{i} (q)$ in the Hausdorff topology
on $\widehat{\Gamma_G}$. Passing to a subsequence if necessary we
can assume that $\lambda_i^r \cap B(i) = \emptyset$. Hence 
$\Lambda_{CT} \subset \Lambda$. $\Box$

\medskip

{\bf Remark:} Suppose $H_1$ is a hyperbolic subgroup of $H$. Let
$\hat{j}$ and $\hat{i}$ denote Cannon-Thurston maps for the pairs
$({H_1},H)$ and $(H,G)$ respectively. Then the composition 
${\hat{i}}\cdot{\hat{j}}$ is a Cannon-Thurston map for the pair
$({H_1},G)$. Further 
from Lemma \ref{equality} it follows that 
\begin{center}

${\Lambda}({H_1},G) = {\Lambda}({H_1},H) \cup {({\hat{j}})}^{-1} ({\Lambda}
(H,G))$.

\end{center}

\medskip

\medskip

\section{Extensions by Free Groups}

Let $1 \rightarrow H \rightarrow G \rightarrow {\mathbb{Z}} \rightarrow 1$
denote an exact sequence of hyperbolic groups arising out of a hyperbolic
automorphism $\phi$  of the hyperbolic group $H$. The
notion of a hyperbolic automorphism was defined in \cite{BF} (see below)
and shown to be equivalent to requiring that $G$ be hyperbolic.

\medskip

{\bf Definition:} {\it Let $\phi$ be an automorphism of a hyperbolic
group $H$ (equipped with the word metric $|.|$). 
Let $\lambda > 1$. Let $S({\phi},{\lambda}) = \{ h \in H :
|{\phi}(h)| > {\lambda}|h| \}$. If $h \in S({\phi},{\lambda})$, we say
${\phi}$ stretches $h$ by $\lambda$. $\phi$ will be called 
{\bf hyperbolic} if for all $\lambda > 1$
there exists $n > 0$ such that for all
$h \in H$, at least one of $\phi^n$ or $\phi^{-n}$ stretches 
$h$ by $\lambda$. }

\medskip

$\phi$ and $\phi^{-1}$ induce bijections (also denoted by $\phi$ and 
$\phi^{-1}$ ) of the vertices of $\Gamma_H$. 

\smallskip

{\it A {\bf free homotopy representative} of a word $w\in{H}$
is a geodesic $[a,a{w_0}]$ 
in $\Gamma_H$ where $w_0$ is a shortest word in the conjugacy
class of $w$ in $H$.}

\smallskip

Given $h\in{H}$ let ${\Sigma}(h,n,+)$ (resp. ${\Sigma}(h,n,-)$
be the ($H$-invariant) collection of
all free homotopy representatives of  ${\phi}^n{(h)}$ (resp. $\phi^{-n}(h)$)
in $\Gamma_H$.  
The intersection with ${\partial}{\Gamma_H}{\times}{\partial}{\Gamma_H}$
of the union 
of all bi-infinite subsequential limits (in the Hausdorff topology on
$\widehat{\Gamma_H}$ )
of elements of  ${\Sigma}(h,n,+)$ (resp. ${\Sigma}(h,n,-)$
as $n\rightarrow{\infty}$ will be denoted by 
${\Lambda}_{+} (h)$ (resp.  ${\Lambda}_{-} (h)$).

\smallskip

{\bf Definition:} {\it The  {\bf stable }
and  {\bf unstable ending laminations} are respectively given by}

\begin{center}

$\Lambda_+$ = ${{\bigcup}_{h\in{H}}}{{\Lambda}_{+}(h)}$ \\
$\Lambda_-$ = ${{\bigcup}_{h\in{H}}}{{\Lambda}_{-}(h)}$ \\

\end{center}

\smallskip

\begin{theorem} \cite{mitra1}
Let $1 \rightarrow H \rightarrow G \rightarrow {\mathbb{Z}} \rightarrow 1$
denote an exact sequence of hyperbolic groups arising out of a hyperbolic
automorphism $\phi$  of the hyperbolic group $H$.
Then there exists a Cannon Thurston map for the pair $(H,G)$.
\label{CT}
\end{theorem}

\begin{theorem} \cite{mitra2}
Let $1 \rightarrow H \rightarrow G \rightarrow {\mathbb{Z}} \rightarrow 1$
denote an exact sequence of hyperbolic groups arising out of a hyperbolic
automorphism $\phi$  of the hyperbolic group $H$.
Then $\Lambda_{CT} = \Lambda_{+} \cup \Lambda_{-}$.
\label{equallams}
\end{theorem}

Further, it is shown  in \cite{mitra2} that only finitely many $h$'s need
be considered in the definition of $\Lambda_{+}$ or $\Lambda_{-}$.

We turn now to the main focus of this paper, the case where
 $H$ is free
and $\phi$ is an irreducible hyperbolic automorphism \cite{BF} .
 Such automorphisms have been
studied in great detail by Bestvina, Feighn and Handel \cite{BF} ,
\cite{BFH} , \cite{BFHa} . 

\medskip

{\bf Definition:} A non negative irreducible matrix is {\it aperiodic}
if it has an iterate that is strictly positive.

\medskip

{\bf Definition:}
Let us assume for a start that transition matrices of $\phi$ and
${\phi}^{-1}$ with respect to train-track representatives [see \cite{BH}
for definitions] are  aperiodic. We shall refer to such automorphisms
as {\it aperiodic}. 
Note that in this definition, we require transition matrices of both
$\phi$ and $\phi^{-1}$ to be aperiodic.

\medskip

In this case, the definitions
of ending laminations here and in \cite{BFH} do not coincide. However the difference between these is now understood \cite{kl1, kl2}.
We recall definitions from \cite{BFH} .

Let $f : X \rightarrow X$ be a train-track representative of an outer
automorphism with aperiodic transition matrix. Endow $X$ with the
structure of a marked $\mathbb{R}$-graph so that $f$ expands lengths of edges
by a uniform factor $\lambda > 1$. Let $x\in X$ be an $f$-periodic point 
in the interior of some edge. Let $\epsilon > 0$ be small, and let $U$
be the $\epsilon$-neighborhood of $x$. Then for some $N > 0$, 
$U \subset {f^N}(U) $. Choose an isometry $l : (-{\epsilon},{\epsilon})
\rightarrow U$ and extend it to the unique locally isometric immersion 
$l : {\mathbb{R}} \rightarrow X$ such that $l({{\lambda}^N}t) = {f^N}(l(t))$.
We say $l$ is obtained by iterating a neighborhood of $x$. $l$ will
also be termed a {\it leaf} of the ending lamination.

{\bf Definitions :} Two isometric immersions $[a,b] \rightarrow X$
and $[c,d] \rightarrow X$ are said to be equivalent if there is
an isometry of $[a,b]$ onto $[c,d]$ making the triangle commute.

A {\it leaf segment} of an isometric immersion 
$\mathbb{R} \rightarrow X$ is the equivalence class of the restriction
to a finite interval. A {\it half leaf} of an isometric immersion 
$\mathbb{R} \rightarrow X$ is the equivalence class of the restriction
to a semi-infinite interval $(-\infty,a]$ or $[b, \infty)$.

Two isometric immersions $l, l^{\prime} : \mathbb{R} \rightarrow X$
are {\it (weakly) equivalent } if every leaf segment of $l$ is
a leaf segment of $l^{\prime}$ and vice versa.

Since $f$ has an aperiodic transition matrix, $l$ is surjective.
Using this, Bestvina, Feighn and Handel \cite{BFH}
show that any two leaves
of the ending lamination obtained by iterating neighborhoods of
$f$-periodic points are equivalent. Let $\Lambda_{BFH}(f)$  denote the collection of leaves obtained from $f$ in this way.

\medskip

\noindent {\bf Relation Between $\Lambda_{BFH}(f)$ and Ending Laminations} \footnote{I had incorrectly equated  $\Lambda_+$ and $\Lambda_{BFH}(f)$
at this stage
of \cite{mahan-scottswar}. The difference lies in the diagonal leaves to be added.} \\
$Diag(\Lambda_{BFH}(f))$ will denote the diagonal extension of $\Lambda_{BFH}(f)$ following \cite{kl1}. 
In the context of surfaces this simply corresponds to adding on the diagonals of an ideal polygon.
We now explain what this is more precisely in the context of free groups.
Define a relation $p \sim q$, if $p, q$ are end-points of a leaf of $\Lambda_{BFH}(f)$. Then the transitive closure of the relation of this relation gives 
$Diag(\Lambda_{BFH}(f))$, i.e. $p, q$ are end-points of a leaf of $Diag(\Lambda_{BFH}(f))$ if there exists a finite sequence $p=p_0, p_1, \cdots ,
p_n = q$ such that $p_i, p_{i+1}$ are end-points of a leaf of $\Lambda_{BFH}(f)$ for all $i=0, \cdots , n-1$.

The following Proposition due to Kapovich and Lustig furnishes the relationship we need between $\Lambda_{BFH}(f)$ and $ \Lambda_+$.

\begin{theorem} (Proposition 6.4 of \cite{kl2}) Let $1 \rightarrow H \rightarrow G \rightarrow {\mathbb{Z}} \rightarrow 1$
be an exact sequence of hyperbolic groups induced by an aperiodic
automorphism $\phi$ of the free group $H$ and let $f$ be a train-track representative of $\phi$.
Then $Diag(\Lambda_{BFH}(f))= \Lambda_+$. \label{kl} \end{theorem}

\bigskip

We can now  use the results of \cite{BFH} and \cite{BFHa}
in our context.

{\bf Remark :} Since any two leaves are (weakly) equivalent in the
sense of \cite{BFH} above, the equivalence class can alternately be
obtained
by translating some (any) leaf by elements of the free group
and taking Hausdorff limits. This is analogous to the case for
surfaces where the stable lamination of a pseudo anosov
diffeomorphism is the closure of some (any) leaf.

The next Proposition follows from  \cite{BFH} (see particularly Propositions 1.6 and 2.4) and Theorem \ref{kl}.
It says roughly that any   leaf of the stable (or unstable) lamination 
of an aperiodic automorphism `fills' the free group $H$.

The proof of Proposition 1.6 of \cite{BFH} shows that any {\it half-leaf}
`fills' the free group $H$. Proposition \ref{kl} shows that any leaf of the stable (or unstable) lamination must contain a 
 {\it half-leaf} of the corresponding lamination $\Lambda_{BFH}(f)$.

\begin{prop}
Let $1 \rightarrow H \rightarrow G \rightarrow {\mathbb{Z}} \rightarrow 1$
be an exact sequence of hyperbolic groups induced by an aperiodic
automorphism $\phi$ of the free group $H$ (i.e. $\phi$ and
$\phi^{-1}$ have aperiodic transition matrices).  
 If $(p,q) \in \Lambda_{+} $ or $\Lambda_{-}$ lie in the boundary 
$\partial{\Gamma_{H_1}} \subset \partial{\Gamma_{H}}$ 
for a finitely generated subgroup $H_1$ of $H$,
then $H_1$ is of finite index in $H$.
\label{BFH}
\end{prop}

\begin{proof} It suffices to show that no leaf of $\Lambda_+$ (or $\Lambda_-$) is carried by a finitely generated subgroup $H_1$ of 
infinite index  $H$. As noted above, 
the proof of Proposition 1.6 of \cite{BFH} shows that any {\it half-leaf}
`fills' the free group $H$, i.e. it cannot have a limit point in a translate of the limit set of $H_1$. Hence no end-point of a leaf of
$\Lambda_{BFH}(f)$ for a train-track representative $f$ of  $\phi$ can have a limit point in a translate of the limit set of $H_1$. 
Since the set of end-points of $\Lambda_{BFH}(f)$ coincide with the end-points of  $\Lambda_+$ by Proposition \ref{kl} we are done.
\end{proof}

We are now in a position to prove 
the main theorem of this paper for aperiodic $\phi$ .

\begin{theorem}
Let $1 \rightarrow H \rightarrow G \rightarrow  {\mathbb{Z}} \rightarrow 1$
be an exact sequence of hyperbolic groups
induced by an aperiodic automorphism $\phi$ of the
free group $H$. 
Let $H_1$ be a finitely generated subgroup of infinite index
 in 
$H$. Then $H_1$ is quasiconvex in $G$.
\label{main}
\end{theorem}

{\bf Proof:} From Theorem \ref{CT} the pair $(H,G)$ has a Cannon -
Thurston map. Further, from Theorem \ref{equallams} 

\begin{center}

$\Lambda_{CT} (H,G) = \Lambda (H,G) = \Lambda_{+} \cup \Lambda_{-}$.

\end{center}

Let $j : H_1 \rightarrow H$ and $i : H \rightarrow G$
denote inclusions. Since
$H_1$ is quasiconvex in $H$,  ${\Lambda}({H_1},H) = \emptyset$
(Lemma \ref{qccrit} ).
Further from Proposition \ref{BFH} above
$j^{-1} ({\Lambda}(H,G)) = \emptyset$.

Also from the remark following
Lemma \ref{equality} , $\Lambda{({H_1},G)} = \Lambda({H_1},H)
\cup j^{-1} ({\Lambda}(H,G)) = \emptyset$.

Hence from Lemma \ref{qccrit} $H_1$ is quasiconvex in $G$. $\Box$

\medskip

As a second step we deal with automorphisms $\phi$ of $H$ satisfying
the following: \\
  There exists a  decomposition $H = {K_1}*{K_2}*{\cdots}*{K_n}$
of $H$ into $\phi$-invariant factors $K_i$ such that the
restrictions ${\phi}|_{K_i} = \phi_i$ are aperiodic. 

Let $\Lambda_i$ denote the ending laminations of $\phi_i$.

We need to first show that the endpoints of {\it half-leaves} of 
$\Lambda\cap{\partial{\Gamma_{K_i}}}$ are precisely the  endpoints of {\it half-leaves} $ \Lambda_i$. 
Let us consider a reduced word $w=a_1\cdots a_k$ such that each $a_i$ is a maximal subword lying in a  
$\phi$-invariant factor $K_i$. We now consider subsequential limits of the geodesic representatives of $\phi^n(w)$ giving us the leaves of
$\Lambda_+$. Since each $a_i$ is a maximal subword lying in a $\phi$-invariant factor, there is no cancellation between the geodesic representatives
of $\phi^n(a_i)$. Also note that since the lengths of each $\phi^n(a_i)$ tends to infinity as $n$ tends to infinity, it suffices
to consider $k=1, 2$. If $k=1$, then the subsequential limits of the iterates $\phi^n(a_1)$ are precisely
the leaves of $\Lambda_{i+}$ and we are done in this case. Let $k=2$ and $w=a_1a_2$.  
Then the subsequential   limits of the iterates $\phi^n(w)= \phi^n(a_1)\phi^n(a_2)$ contain half leaves of limits of
the iterates $\phi^n(a_1)$ or $\phi^n(a_2)$ or both. In any  case the  endpoints of {\it half-leaves} of 
$\Lambda\cap{\partial{\Gamma_{K_i}}}$ are precisely the  endpoints of {\it half-leaves} $ \Lambda_i$.

Suppose  $H_1$ is a finitely generated subgroup of $H$ that is
distorted
in $G$. Since $H_1$ is quasi-convex in $H$, there exists
a pair $(p,q) \in \Lambda = \Lambda_{+} \cup \Lambda_{-}$ 
lying on the boundary 
$\partial{\Gamma_{H_1}} \subset \partial{\Gamma_{H}}$. 
Let $l$ be a leaf of $\Lambda$
 joining $p,q$. By Theorem \ref{equallams} $l$ lies
in the Hausdorff limit of sequences of segments obtained by iterating
$\phi$ or $\phi^{-1}$ on some $h\in{H}$.
 By the pigeon-hole principle there exists arbitrarily long segments
of $l$ contained in a (fixed) conjugate of $K_i$ for some $i$. 
For ease of exposition let us assume that this is the trivial
conjugate of $K_i$, i.e. $K_i$ itself. Translating
by appropriate elements of $H_1$ and taking a Hausdorff limit
we obtain an endpoint of a leaf (or half-leaf) of the ending lamination $\Lambda$
lying in the intersection $\partial{\Gamma_{H_1}}{\cap}\partial{\Gamma_{K_j}}$.
In particular, there exists a  point $s \in
\partial{\Gamma_{H_1}}$ which is also an end-point of a half-leaf of
$\Lambda_j$ (by the argument in the previous paragraph).

Since intersection of quasiconvex subgroups is quasiconvex
 \cite{short} it follows that $H_1\cap{K_j}$ is quasi-convex
in $H$. In
particular $H_1\cap{K_j}$ is 
finitely generated. Also as observed above, $H_1$ has a boundary point in common with a leaf
of the ending lamination $\Lambda_j$. Hence from Theorem \ref{main} 
$H_1\cap{K_j}$ is a finite index subgroup of $K_j$.

We have shown :

\begin{theorem}
Let $1 \rightarrow H \rightarrow G \rightarrow  {\mathbb{Z}} \rightarrow 1$
be an exact sequence of hyperbolic groups
induced by an  automorphism $\phi$ of the
free group $H$ satisfying the following : \\
  There exists a  decomposition $H = {K_1}*{K_2}*{\cdots}*{K_n}$
of $H$ into $\phi$-invariant factors $K_i$ such that the
restrictions ${\phi}|_{K_i} = \phi_i$ are aperiodic. 

Let $H_1$ be a finitely generated subgroup of infinite index
 in 
$H$ such that $H_1$ is distorted in $G$. Then 
there exist $h\in{H}$ and $K_j$ such that  $H_1$ contains
a finite index subgroup of 
$h^{-1}K_j{h}$. 
\label{main2}
\end{theorem}

\section{Erratum}

In Theorem 3.6 in \cite{mahan-scottswar} I had misquoted Corollary 4.7 of \cite{BFHa}. Hence 
Theorem 3.7 in \cite{mahan-scottswar} stands non-proven.

\medskip
\medskip

{\bf Acknowledgments :} The argument for the case of an aperiodic automorphism
had been outlined in \cite{mitra} . The author would like to thank
his  advisor Andrew
Casson and John Stallings for their help.

{\bf Additional Acknowledgment:}
I would like to thank Ilya Kapovich for pointing out the errors in the published version \cite{mahan-scottswar}
and suggesting the corrections here.

\bibliography{scottswar}

\begin{thebibliography}{10}

\bibitem{BF}
M.~Bestvina and M.~Feighn.
\newblock A {C}ombination theorem for {N}egatively {C}urved {G}roups.
\newblock {\em J. Diff. Geom., vol 35}, pages 85--101, 1992.

\bibitem{BFHa}
M.~Bestvina, M.~Feighn, and M.~Handel.
\newblock The {T}its' alternative for {Out($F_n$) I}: Dynamics of exponentially
  growing automorphisms.
\newblock {\em preprint, http://arxiv.org/pdf/math/9712217v1.pdf}.

\bibitem{BFH}
M.~Bestvina, M.~Feighn, and M.~Handel.
\newblock Laminations, trees and irreducible automorphisms of free groups.
\newblock {\em GAFA vol.7 No. 2}, pages 215--244, 1997.

\bibitem{BH}
M.~Bestvina and M.~Handel.
\newblock Train tracks and automorpfisms of free groups.
\newblock {\em Ann. Math. 135}, pages 1--51, 1992.

\bibitem{CT}
J.~Cannon and W.~P. Thurston.
\newblock Group {I}nvariant {P}eano {C}urves.
\newblock {\em preprint}.

\bibitem{farb}
B.~Farb.
\newblock The extrinsic geometry of subgroups and the generalized word problem.
\newblock {\em Proc. LMS (3) 68}, pages 577--593, 1994.

\bibitem{Gromov2}
M.~Gromov.
\newblock Asymptotic {I}nvariants of {I}nfinite {G}roups.
\newblock {\em in Geometric Group Theory,vol.2; Lond. Math. Soc. Lecture Notes
  182 (1993), Cambridge University Press}.

\bibitem{Gromov}
M.~Gromov.
\newblock Hyperbolic {G}roups.
\newblock {\em in Essays in Group Theory, ed. Gersten, MSRI Publ.,vol.8,
  Springer Verlag,1985}, pages 75--263.

\bibitem{kl1}
I.~Kapovich and M.~Lustig.
\newblock Invariant laminations for irreducible automorphisms of free groups.
\newblock {\em preprint, arXiv:1104.1265}.

\bibitem{kl2}
I.~Kapovich and M.~Lustig.
\newblock On the fibers of the {Cannon-Thurston} map for free-by-cyclic groups.
\newblock {\em preprint, arXiv:1207.3494}.

\bibitem{mitra2}
M.~Mitra.
\newblock Ending {L}aminations for {H}yperbolic {G}roup {E}xtensions.
\newblock {\em GAFA vol.7 No. 2}, pages 379--402, 1997.

\bibitem{mitra}
M.~Mitra.
\newblock Maps on boundaries of hyperbolicmetric spaces.
\newblock {\em Ph{D} {T}hesis, {U}.{C}.{B}erkeley}, 1997.

\bibitem{mitra1}
M.~Mitra.
\newblock Cannon-{T}hurston {M}aps for {H}yperbolic {G}roup {E}xtensions.
\newblock {\em Topology}, 1998.

\bibitem{mahan-scottswar}
M.~Mitra.
\newblock {On a Theorem of Scott and Swarup}.
\newblock {\em Proc. Amer. Math. Soc. 127 (6)}, pages 1625--1631, 1999.

\bibitem{scott}
P.~Scott.
\newblock Subgroups of surface groups are almost geometric.
\newblock {\em Journal L.M.S. 17}, pages 555--65, 1978.

\bibitem{scottswar}
P.~Scott and G.~Swarup.
\newblock Geometric {F}initeness of {C}ertain {K}leinian {G}roups.
\newblock {\em Proc. AMS 109}, pages 765--768, 1990.

\bibitem{short}
H.~Short.
\newblock Quasiconvexity and a theorem of {H}owson's.
\newblock {\em Group {T}heory from a {G}eometrical {V}iewpoint ({E}. {G}hys,
  {A}. {H}aefliger, {A}. {V}erjovsky eds.)}, 1991.

\end{thebibliography}
\bibliographystyle{plain}
\noindent        University of California at Berkeley, CA 94720 \\
         email: mitra@math.berkeley.edu \\
\noindent Current Address: School of Mathematical Sciences, RKM Vivekananda University, P.O. Belur Math, Dt. Howrah,
                   West Bengal -711202, India. \\
\end{document}